\documentclass[11pt]{amsart}
\usepackage[utf8x]{inputenc}
\usepackage[T1]{fontenc}
\usepackage[a4paper,includeheadfoot,margin=2.54cm]{geometry}
\usepackage{stmaryrd}
\usepackage{amsthm}
\usepackage{mathtools}
\mathtoolsset{showonlyrefs}
\usepackage[all]{xy}
\usepackage{xcolor}
\usepackage{enumitem}
\setenumerate[1]{itemsep=0pt,partopsep=0pt,parsep=\parskip,topsep=5pt}
\setitemize[1]{itemsep=0pt,partopsep=0pt,parsep=\parskip,topsep=5pt}
\setdescription{itemsep=0pt,partopsep=0pt,parsep=\parskip,topsep=5pt}
\usepackage{times}
\usepackage[scaled=.92]{helvet}
\usepackage{mathabx}
\usepackage[cal=esstix,scr=rsfs,bb=ams,frak=euler]{mathalfa}
\usepackage{amsrefs}
\usepackage{graphicx}
\usepackage[
colorlinks=true,
linkcolor=purple,
citecolor=blue,
hyperindex,
plainpages=false,
bookmarks=true,
bookmarksopen=true,
bookmarksnumbered]{hyperref}
\theoremstyle{plain}
\newtheorem{theorem}{Theorem}[section]
\newtheorem*{theorem*}{Theorem}

\newtheorem{lemma}[theorem]{Lemma}

\theoremstyle{definition}

\theoremstyle{remark}

\numberwithin{equation}{section}

\newcommand{\wc}{\rightharpoonup} 
\newcommand{\oball}{\mathbf{U}}
\newcommand{\cball}{\mathbf{B}}
\newcommand{\HM}{\mathcal{H}}
\newcommand{\ud}{\ensuremath{\,\mathrm{d}}}

\newcommand{\cf}{\mathscr{C}}
\DeclareMathOperator{\spt}{spt}
\DeclareMathOperator{\id}{id}
\DeclareMathOperator{\Tan}{Tan}
\DeclareMathOperator{\Lip}{Lip}
\DeclareMathOperator{\dist}{dist}
\DeclareMathOperator{\diam}{diam}
\DeclareMathOperator{\ap}{ap} 

\newcommand{\mr}{\mathop{\vrule height 1.6ex depth 0pt width
0.13ex\vrule height 0.13ex depth 0pt width 1.3ex}\nolimits}
\newcommand{\cech}{\check{\mathsf{H}}}
\newcommand{\cc}{\check{\mathscr{C}}}
\begin{document}

\title{On the Reifenberg Plateau Problem in Hilbert space}
\author{Yangqin FANG}
\address{\parbox{\linewidth}{\strut Yangqin Fang\\ School of Mathematics and Statistics\\
Huazhong University of Science and Technology\\
430074, Wuhan, P.R. China\strut}}
\email{yangqinfang@hust.edu.cn}

\date{}
\maketitle

\begin{abstract}
	In this paper, we will solve the Reifenberg Plateau Problem in Hilbert
	space. 
\end{abstract}

\section{Introduction}
Plateau's problem is to show the existence of a minimal 
surfaces with a given boundary. There are various variants of Plateau's problem,
which arise from different understanding of the surfaces. The classical form of 
the problem was solved independently by J. Douglas \cite{Douglas:1931 } and
T. Rad\'o \cite{Rado:1930}, in which the surfaces are understood as
parametrizations of the unit disc in $\mathbb{R}^{3}$.

Federer and Fleming introduced currents \cite{FF:1960} in $\mathbb{R}^n$, and
the surfaces are understood as integral currents. The Plateau's problem can be
characterized as minimizing the mass or the size of integral currents with a 
given boundary. Indeed, mass-minimization problem immediately follows from a
compactness theorem, and size-minimization problem is still open. The notion 
of currents was extended to arbitrary metric spaces by Ambrosio and Kirchheim
in \cite{AK:2000c}, and mass-minimization problem in metric space was solved
thereafter.

Reifenberg \cite{Reifenberg:1960} used \v{C}ech homology to depict the
Plateau's problem. The surfaces are directly understood as compact subsets in
$\mathbb{R}^{n}$, and we are going to minimize the $d$-dimensional Hausdorff
measure. But one difficulty is about how to understand the boundary, since
there is no canonical boundary operator such as the boundary operator for
currents. Indeed the difficulty can be overcame by \v{C}ech homology. \v{C}ech
homology can be certainly replaced by other homology to characterize the
Plateau's problem, however new difficulties will occur and are hard to be
overcame when we try to find minimizers. Actually, Reifenberg proved that in
euclidean space there exist minimizers when the coefficient group is a compact
abelian group, recent development indicate that it can be extended to any
abelian group and even on a closed $C^2$ submanifold of $\mathbb{R}^{n}$, see
for example \cite{Pauw:2009,Fang:2013,Fang:2020p}. There are two main
obstacles, which obstruct the way to develop the Plateau's problem in Hilbert
space or Banach space. First, Federer-Fleming projection is a useful tool to
deal with the rectifiablity of pre-minimizers, but it vanished in an infinitely
dimensional space. Second, in a Banach space even of finite dimension, there
does not need to exist a projection that reduce the Hausdorff measure, and that 
gives a lot of difficulties when we locally do the projection of a set onto its
tangent plane. In this paper, we will go around the former, and that allow us 
to solve the Plateau's problem in Hilbert space. Let us start with Reifenberg's
frames, but in Hilbert space.

Let $H$ be a Hilbert space, $B\subseteq H$ be a compact set, $d$ be a positive
integer, $G$ be an abelian group, $L$ be a subgroup of \v{C}ech homology group
$\cech_{d-1}(B;G)$. A compact set $E\subseteq H$ is called spanning $L$ (or with
algebraic boundary containing $L$) if $E\supseteq B$ and
$\cech_{d-1}(i_{B,E})(L)=0$, where $i_{B,E}: B\to E$ is the inclusion mapping and
$\cech_{d-1}(i_{B,E})$ is the homomorphism induced by $i_{B,E}$. We denote by
$\cc_d(H,B,L)$ the collection of all compact subsets $E$ which are spanning $L$.
It is easy to see that $\cc_d(H,B,L)$ is nonempty. If $E\in \cc_d(H,B,L)$ and
$\HM^d(E\setminus B) = \inf\{\HM^d(X\setminus B):X\in \cc_d(H,B,L)\}$, then we
call $E$ a minimizer for $\cc_d(H,B,L)$. The Reifenberg Plateau Problem is to
show the existence of minimizers for $\cc_d(H,B,L)$, i.e. the solutions to the
Plateau's problem. In above settings, we do not exclude the case that $H$ is
of finite dimension.
\begin{theorem} \label{thm:PP}
	Let $H$ be a Hilbert space, $B\subseteq H$ be a compact set, $d$ 
	be a positive integer. Let $G$ be an abelian group, $L\subseteq 
	\check{H}_{d-1}(B;G)$ be a subgroup. Then there exist minimizers for 
	$\cc_d(H,B,L)$.
\end{theorem}

\section{Preliminary}
Let $(X,\rho)$ be a metric space. For any $Y\subseteq X$ and $x\in X$, we
define the distance from $x$ to $Y$ by  
\[
	\dist(x,Y)=\inf\big\{\rho(x,y):y\in Y\big\}.
\]
For any set $E,F\subseteq X$, the Hausdorff distance $d_{\HM}(E,F)$ is defined by 
\[
	d_{\HM}(E,F) = \max\left\{\sup_{x\in E}\dist(x,F),\sup_{y\in F}\dist(y,E)
	\right\}.
\]
Let $\mathscr{C}_X$ be the collection of compact subset in $X$. If $(X,\rho)$ 
is complete, then $(\mathscr{C}_X,d_{\HM})$ is complete; if $(X,\rho)$ is 
compact, then $(\mathscr{C}_X,d_{\HM})$ is compact. For any set $Y\subseteq X$,
the diameter of $Y$ is defined by 
\[
\diam(Y)=\sup\{\rho(y,z):y,z\in Y\}.
\]
 For any set $E\subseteq X$
and number $s\geq 0$, the $s$-dimensional Hausdorff measure $\HM^s(E)$ is
defined by
\[
	\HM^s(E) = \lim_{\delta\to 0+} \inf\left\{\sum_{i=1}^{\infty}\diam(A_i)^s:
	E\subseteq \bigcup_{i=1}
	^{\infty}A_i,\diam(A_i)\leq \delta\right\}.
\]

We denote by $\cf_c(X,\mathbb{R})$ the set of continuous
functions from $X$ to $\mathbb{R}$ which have compact support, and denote by
$\mathscr{R}(X)$ the set of all Radon measures on $X$, which is always endowed
with weak topology given by saying that $\mu_k\wc \mu$ if and only if
$\mu_k(f)\to \mu(f)$ for any $f\in \cf_c(X,\mathbb{R})$. Indeed, the weak
topology is the topology generated by sets
\[
	\{\mu\in \mathscr{R}(X): a<\mu(f)<b\}, \ a, b\in \mathbb{R}, \ f\in
	\cf_c(X,\mathbb{R}).
\]
Let $\{\mu,\mu_k\}$ be a sequence of Radon measure on $X$. If $X$ is locally
compact and $\mu_k\wc \mu$, then for any open set $O\subseteq X$ and compact 
set $K\subseteq X$, we have 
\[
	\mu(O)\leq \liminf_{k\to \infty}\mu_k(O)
\]
and 
\[
	\mu(K)\geq \limsup_{k\to \infty} \mu_k(K).
\]
If $X$ is locally compact, then for any $c>0$, $\{\mu\in \mathscr{R}(X): 
\mu(X)\leq c\}$ is compact with respect to the weak topology.

For any $x\in X$ and $r>0$,
we denote by $\oball(x,r)$, $\cball(x,r)$ and $\partial \cball(x,r)$ the open
ball, the closed ball and the sphere centered at $x$ of radius $r$ respectively.
For any Radon measure $\mu$ on $X$, positive integer $d$, $x\in X$ and $r>0$, 
we define 
\[
	\Theta_{\mu}^d(x,r)=\frac{\mu(\cball(x,r))}{\omega_d r^d},
\]
where $\omega_d$ is the $d$-dimensional Hausdorff measure of the unit ball in
the euclidean space $\mathbb{R}^d$. The $d$-dimensional upper density and
lower density of $\mu$ at $x$ are defined by 
\[
	\overline{\Theta}_{\mu}^d(x) = \limsup_{r\to 0+}\Theta_{\mu}^d(x,r)
\]
and 
\[
	\underline{\Theta}_{\mu}^d(x) = \liminf_{r\to 0+}\Theta_{\mu}^d(x,r)
\]
respectively. If the two coincide, then the common value is called the density
of $\mu$ at $x$, and we denote it by $\Theta_{\mu}^d(x)$.
For any set $E\subseteq X$, $\Theta_E^d(x,r)$ and
$\Theta_E^d(x)$ are respectively defined as $\Theta_{\mu}^d(x,r)$ and
$\Theta_{\mu}^d(x)$ with $\mu= \HM^d\mr E$. If $E\subseteq X$ is
$\HM^d$-measurable and $\HM^d(E)<\infty$, then $2^{-d}\leq \overline{\Theta}
_{E}^d(x)\leq 1$ for $\HM^d$-a.e. $x\in E$, see \cite[Remark 3.7]{Simon:1983}. 
If $E\subseteq X$ is $\HM^d$-measurable and
$d$-rectifiable, then $\Theta_{E}^d(x)=1$ for $\HM^d$-a.e. $x\in E$, see
\cite{Kirchheim:1994}.

A set $E\subseteq X$ is called $d$-rectifiable, if there exits a sequence of
Lipschitz mapping $f_i:\mathbb{R}^d\to X$ such that 
\[
	\HM^d \left(E\setminus \bigcup_{i=1}^{\infty} f_i\big(\mathbb{R}^d\big)\right)=0.
\]
A set $F\subseteq X$ is called purely $d$-unrectifiable if for any
$d$-rectifiable set $E\subseteq X$, $\HM^d(F\cap E)=0$. For any
$\HM^d$-measurable set $E\subseteq X$ of finite $\HM^d$ measure, it can be
decomposed as the union of a $d$-rectifiable $\HM^d$-measurable set and a purely
$d$-unrectifiable $\HM^d$-measurable set, i.e. $E=E^{irr} \sqcup E^{rec}$, the
decomposition is indeed unique up to a set of $\HM^d$ measure zero. 

\begin{lemma} [$5r$-covering lemma]
	Let $(X,\rho)$ be a metric space, $\mathcal{B}$ be a family of open balls in
	$X$ such that $\sup\{\diam(B):B\in \mathcal{B}\}<\infty$. Then there is a
	subfamily $\mathcal{B}'\subseteq \mathcal{B}$ consisting of mutually disjoint
	balls such that 
	\[
		\bigcup_{B\in \mathcal{B}} B \subseteq \bigcup_{B\in \mathcal{B}'} 5 B.
	\]
	In addition, if $X$ is separable, or there is a finite Borel measure $\mu$
	on $X$ such that $\mu(B)>0$ for all $B\in \mathcal{B}$, then $\mathcal{B}'$ is countable.
\end{lemma}
The proof can be found for example in \cite{Simon:1983,Mattila:1995, Heinonen}. As a
direct consequence of the above lemma, we have the following lemma which we
may use later.
\begin{lemma}
Let $\mu$ be a Radon measure on $X$, $0<t<\infty$, and $A\subseteq X$ be a 
$d$-rectifiable Borel set.
\begin{itemize}
	\item If $\overline{\Theta}_{\mu}^d(x)\geq t$ for all $x\in A$, then $\mu(A)
		\geq t \HM^d(A)$.
	\item If $\underline{\Theta}_{\mu}^d(x)\leq t$ for all $x\in A$, then $\mu(A)\leq t
		\HM^d(A)$.
\end{itemize}
\end{lemma}
The proof of the lemma can be found in \cite[Theorem 3.2]{Simon:1983}.

Grothendieck \cite[Lemme 12]{Grothendieck:1955} observed a signifiant 
consequence of a  paper by Schwartz and Dieudonn\`e, that gives a convenient
characterization of compact subsets in a Banach space. That is indeed our key
point to reduce  the Plateau's problem in a infinite dimensional to euclidean
spaces.
\begin{lemma}\label{le:Gr}
	Let $W$ be a Banach space, $K\subseteq W$ be a closed set. Then $K$ is
	compact if and only if there exists a sequence $\{y_n\}_{n=1}^{\infty}
	\subseteq W$ with  $\|y_n\|\to 0$ such that $K$ is contained in the closed
	convex hall of
	$Y=\{y_n: n\geq 1\}$. 
\end{lemma}

Let $H$ be a Hilbert space, $V$ be a subspace of $H$, we denote by $V_{\sharp}$
the orthogonal projection from $H$ onto $V$, then $\Lip(V_{\sharp})=1$, thus
$\HM^d(V_{\sharp}(E))\leq \HM^d(E)$ for any $E\subseteq H$. We put $V_{\sharp}
^{\perp}= \id_{H} - V_{\sharp}$. For any set $E\subseteq H$ and $x\in E$, a
$d$-plane $T$ is called an approximate tangent plane of $E$ at $x$, if 
\[
	\lim_{r\to 0+}\frac{\HM^d(E\cap \cball(x,r)\setminus \mathcal{C}(T,x,
	\delta))}{\omega_d r^d}=0,\ \forall \delta>0,
\]
where $\mathcal{C}(T,x,\delta)=\big\{z\in H:\|T_{\sharp}^{\perp}(z-x)\|\leq \delta
\|z-x\|\big\}$. It is easy to see that if the approximate tangent planes exits,
then it is unique, we denote it by $\Tan^d(E,x)$. Indeed, if $E$ is
$d$-rectifiable, then $\Tan^d(E,x)$ exist for $\HM^d$-a.e. $x\in E$.

For any $d$-rectifiable set $E\subseteq \mathbb{R}^n$, and Lipschitz mapping 
$f:E\to \mathbb{R}^m$, we denote by $\ap Df$ and $\ap J_d f$ the
$d$-approximately differential and $d$-approximately jacobian of $f$, see 
Theorem 3.2.19, Corollary 3.2.20 and Theorem 3.2.22 in \cite{Federer:1969} for
details.

\section{Existence of minimizers}
In this section, we let $B\subseteq H$ be a fixed compact set, $G$ be a fixed
abelian group. Let $L$ be a subgroup of $\cech(B;G)$. By Lemma \ref{le:Gr},
there is a sequence $\{y_n\}_{n=1}^{\infty}\subseteq H$ such that
$\|y_n\|\to 0$ and $B$ is contained in the closed convex hall of $Y=\{y_n:n\geq
1\}$. Let $H'$ be the closure of the subspace which generated by $Y$. Then
$H'$ is separable. Let $V_n$ be the subspace spanned by $\{y_1,\cdots,y_n\}$, 
$\pi_n= (V_n)_{\sharp}$ be the orthogonal projection from $H$ to $V_n$. Put
$B_n=\pi_n(B)$ and $L_n= \check{H}(\pi_n)(L)$. Then $B_n$ is a compact subset
in $V_n$. We denote by $C$ the closed convex hall of $Y$, and put
$\delta_n=\sup\{\dist(x,V_n):x\in C\}$. Then $C\subseteq H'$, $C$ is compact,
$\{\mu\in \mathscr{R}(H):\spt \mu \subseteq C, \mu(H)\leq \alpha\}$ is weakly
compact for any $\alpha>0$, and $d_{\HM}(B_n,B) \leq \delta_n\to 0$. 
\begin{theorem}
	Let $V\subseteq H$ be a finite dimensional subspace, $G$ be an abelian
	group. Then for any compact set $K\subseteq V$ and subgroup $L\subseteq
	\cech(K;G)$, both $\cc_d(V,K,L)$ and $\cc_d(H,K,L)$ have minimizers.
\end{theorem}
\begin{proof}
	The proof for existence of minimizers in $\cc_d(V,K,L)$ can be found in
	\cite[Theorem 1.2]{Fang:2013} or \cite[Corollary 1.3]{Fang:2020p}. 

	Since $K\subseteq V\subseteq H$, we get that $\cc_d(V,K,L)\subseteq \cc_d(H,K,L)$
	and for any $E\in \cc_d(H,K,L)$, $V_{\sharp}(E)\in \cc_d(V,K,L)$. Since
	$\HM^d(V_{\sharp}(E))\leq \HM^d(E)$, we get that 
	\[
		\inf\{\HM^d(F):F\in \cc_d(V,K,L)\}\leq \inf\{\HM^d(F):F\in \cc_d(H,K,L)\}
		\leq \inf\{\HM^d(F):F\in \cc_d(V,K,L)\}.
	\]
	Thus $\inf\{\HM^d(F):F\in \cc_d(V,K,L)\}=\inf\{\HM^d(F):F\in \cc_d(H,K,L)\}$, and
	any minimizer for $\cc_d(V,K,L)$ is also a minimizer for $\cc_d(H,K,L)$.
\end{proof}

\begin{lemma}
	\label{prop:de}
	Let $U\subseteq \mathbb{R}^n$ be an open set, $E\subseteq U$ be a relatively
	closed $d$-rectifiable set such that $\HM^d\mr E$ is locally finite. Suppose 
	that there exist two non-decreasing functions  $M:(0,\infty)\to [1,\infty]$
	and $\varepsilon:(0,\infty)\to [0,\infty]$ such that for any $\cball(x,r)
	\subseteq U$ and Lipschitz mapping $\varphi:U\to U$ with $\spt(\varphi) 
	\subseteq \cball(x,r)$,
	\[
		\HM^d(E\cap \cball(x,r))\leq M(r)\HM^d(\varphi(E\cap \cball(x,r)))+
		\varepsilon(r). 
	\]
	Then for any $x\in U$, by setting $v(r)=\HM^d(E\cap \cball(x,r))$ and
	$r_x=\dist(x,\mathbb{R}^n\setminus U)$, we have that 
	\begin{equation}\label{eq:de1}
		v(r)\leq \frac{rM(r)v'(r)}{d} + \varepsilon(r),
	\end{equation}
	for $\HM^1$-a.e. $r\in (0,r_x)$. In particular, if $M\equiv 1$ and
	$\varepsilon\equiv 0$, i.e. $E$ is minimal in $U$, then $r^{-d}v(r)$ is a
	non-decreasing function for $r\in (0,r_x)$.
\end{lemma}
\begin{proof}
	For any $0<\delta<1$, let $\eta$ be the function defined by 
	\[
		\eta_{\delta}(t)=\begin{cases}
			0,&t\leq 1-\delta,\\
			\frac{t-(1-\delta)}{\delta},& 1-\delta<t<1,\\
			1,&t\geq 1.
		\end{cases}
	\]
	For any $\cball(x,r)\subseteq U$, let $\varphi_{\delta}:U\to U$ be the 
	Lipschitz mapping defined by 
	\[
		\varphi_{\delta}(z)= x+\eta_{\delta}(\|z-x\|/r)(z-x).
	\]
	Let $\psi:U\to \mathbb{R}$ be the function given by $\psi(z)=\|z-x\|$. Then
	for any vector $v\in \mathbb{R}^n$, 
	\[
		D\psi(z)v= \left\langle \frac{z-x}{\|z-x\|},v \right\rangle.
	\]

	For any $z\in E\cap \oball(x,r)\setminus \cball(x,(1-\delta)r)$, if $\Tan^d(E,z)$
	exists, then for any $v\in \Tan^d(E,z)$,
	\[
		\ap D \left(\varphi_{\delta}\vert_E\right)(z) v = \eta_{\delta}
		\left(\frac{\|z-x\|}{r}\right)v 
		+ \eta_{\delta}'\left(\frac{\|z-x\|}{r}\right)\frac{1}{r}\frac{\langle z-x,v \rangle}{\|z-x\|}
		(z-x)
	\]
	We put $T=\Tan^d(E,z)$, $w_1=T_{\sharp}(z-x)$ and $w_2= (z-x)-T_{\sharp}(z-x)
	$. Let $v_1,\cdots,v_d$ be an orthonormal basis of $T$ such that $w_1$ is
	parallel to $v_1$. Then $w_2$ is perpendicular to $T$. Thus
	\[
		\ap D \left(\varphi_{\delta}\vert_E\right)(z) v_1 = \eta_{\delta}
		\left(\frac{\|z-x\|}{r}\right)v_1 + \frac{1}{\delta r}\frac{\langle w_1,
		v_1\rangle}{\|z-x\|}(w_1+w_2), 
	\]
	and for $2\leq i\leq d$,
	\[
		\ap D \left(\varphi_{\delta}\vert_E\right)(z) v_i = \eta_{\delta}
		\left(\frac{\|z-x\|}{r}\right)v_i.
	\]
	Suppose that $w_1 = t v_1$, $\beta = \|w_2\|$, $a=\eta_{\delta}
	(\|z-x\|/r)$ and $b=t/(\|z-x\|\delta r)$, then  
	\[
		\begin{aligned}
			\ap J_d(\varphi_{\delta}\vert_E)(z)&= \|\wedge_d \ap D
			\left(\varphi_{\delta}\vert_E\right)(z) \| = \|\ap D \left(\varphi_{\delta}
			\vert_E\right)(z) v_1 \wedge \cdots \wedge \ap D \left(\varphi_{\delta}
			\vert_E\right)(z) v_d\|\\
			&=\left\|(a + t b)a^{d-1}v_1\wedge\cdots\wedge
			v_d+ b a^{d-1} w_2\wedge v_2\wedge \cdots \wedge v_d\right\|\\
			& = a^{d-1}\left((a+tb)^2+(b\beta)^2\right)^{1/2}.
		\end{aligned}
	\]
	Since $t^2+\beta^2=\|z-x\|^2$ and $tb = t^2/(\|z-x\|\delta r)\leq
	|t|/(\delta r)$, we get that 
	\[
		\ap J_d(\varphi_{\delta}\vert_E)(z)= a^{d-1}\left(a^2+2atb+(t/\delta r)
		^2\right)^2\leq a^{d-1}\left(a+|t|/(\delta r)\right)\leq 1+ a^{d-1}
		|t|/(\delta r).
	\]
	Denote by $\theta(z)\in [0,\pi/2]$ the angle between $z-x$ and $T$. We have
	that 
	\[
		\ap J_d(\varphi_{\delta}\vert_E)(z)\leq 1+ \left(\eta_{\delta}
		\left(\frac{\|z-x\|}{r}\right)\right)^{d-1}\frac{\|z-x\|\cos \theta(z)}
		{\delta r}
	\]
	and $\ap J_d(\psi\vert_E)(z)=\cos \theta(z)$.
		Hence, by setting $E_{r,\delta}= E\cap
	\oball(x,r)\setminus \cball(x,(1-\delta)r)$ and $f(s) = \HM^{d-1}(E\cap
	\partial \cball(x,s))$, we have that 
	\[
		\begin{aligned}
			\HM^d(\varphi_{\delta}(E\cap \oball(x,r)))&=\HM^d(\varphi_{\delta}(E\cap
			\oball(x,r)\setminus \cball(x,(1-\delta)r)))\\
			&\leq \int_{z\in E_{r,\delta}}\ap J_d(\varphi_{\delta}\vert_E)(z) \ud
			\HM^d(z)\\
			&\leq \HM^d(E_{r,\delta}) + \int_{z\in E_{r,\delta}}\left(\eta_{\delta}
			\left(\frac{\|z-x\|}{r}\right)\right)^{d-1}\frac{\|z-x\|\cos \theta(z)}
			{\delta r}\ud \HM^d(z)\\
			&=\HM^d(E_{r,\delta}) +  \frac{1}{\delta r}\int_{(1-\delta)r}^{r}
			s\cdot \left(\eta_{\delta}(s/r) \right)^{d-1}  \HM^{d-1} (E\cap \partial
			\cball(x,s)) \ud s\\
			&=\HM^d(E_{r,\delta}) +  \frac{r}{\delta }\int_{(1-\delta)}^{1}
			s\cdot \left(\eta_{\delta}(s) \right)^{d-1} f(rs) \ud s.
		\end{aligned}
	\]

	Since $v(r)$ is a non-decreasing function, we get that $v'(r)$ exists for 
	$\HM^1$-a.e. $r\in (0,r_x)$. Since $\HM^d\mr E$ is locally finite, we get 
	that $\HM^d(E\cap \partial\cball(x,r))=0$ for $\HM^1$-a.e. $r\in (0,r_x)$.
	Since $f(s)$ is a measurable function and Lebesgue integrable, we get that
	for $\HM^1$-a.e. $s\in (0,r_x)$, $s$ is a Lebesgue point of $f$. Let
	$\mathcal{R}_x$ be the set of points $r\in (0,r_x)$ such that $v'(r)$ exists,
	$\HM^d(E\cap \partial \cball(x,r))=0$ and $r$ is a Lebesgue point of $f$.
	Then $\HM^1((0,r_x)\setminus R_x)=0$.

	Suppose $r\in \mathcal{R}_x$. Then $\HM^d(E\cap \partial \cball(x,r))=0$,
	thus $\HM^d(E_{r,\delta})\to 0$ as $\delta \to 0+$. Since 
	\[
		\frac{1}{\delta}\int_{(1-\delta)}^{1}(\eta_{\delta}(s))^{d-1}ds =
		\frac{1}{d} 
	\]
	and 
	\[
		\lim_{\delta\to 0+}\frac{1}{\delta} \int_{1-\delta}^{1} |f(rs)-f(r)|\ud s
		=0,
	\]
	we get that
	\[
		\limsup_{\delta\to 0+}\HM^d(\varphi_{\delta}(E\cap \cball(x,r)))\leq \frac{r}
		{d}f(r),
	\]
	thus
	\[
		\begin{aligned}
			v(r)&=\HM^d(E\cap \cball(x,r))\leq \limsup_{\delta\to 0+}
			\Big(\HM^d(E_{r,\delta})+M(r)\HM^d(\varphi_{\delta}
			(E\cap \cball(x,r))) + \varepsilon(r)\Big) \\
			&\leq \frac{rM(r)f(r)}{d} +
			\varepsilon(r).
		\end{aligned}
	\]
	Since 
	\[
		v(r)-v((1-\delta)r)= \HM^d(E_{r,\delta})\geq \int_{z\in E_{r,\delta}}\cos
		\theta(z) \ud \HM^d(z) = \int_{(1-\delta)r}^{r} f(s) \ud s,
	\]
	we get that 
	\[
		v'(r) = \lim_{\delta\to 0}\frac{v(r)-v((1-\delta)r)}{\delta}\geq f(r).
	\]
	Hence
	\[
		v(r)\leq \frac{rM(r)v'(r)}{d} + \varepsilon(r)
	\]
	for any $r\in \mathcal{R}_x$, thus \eqref{eq:de1} holds for $\HM^1$-a.e. 
	$r\in (0,r_x)$. If $M\equiv 1$ and $\varepsilon\equiv 1$,
	then $v(r)\leq d^{-1}rv'(r)$ for $\HM^1$-a.e. $r\in (0,r_x)$. Thus 
	\[
		\left(r^{-d}v(r)\right)' = r^{-d-1}(rv'(r)-d\cdot v(r))\geq 0.
	\]
	And $r^{-d}v(r)$ is a non-decreasing function for $r\in (0,r_x)$.

\end{proof}

In the rest of the section, we only consider the case $\inf\{\HM^d(E\setminus B):E\in
\cc_d(H,B,L)\}<\infty$, thus we also have $\inf\{\HM^d(E\setminus B):E\in
\cc_d(V_n,B_n,L_n)\}<\infty$.
\begin{lemma}\label{le:md}
	Let $X_n$ be a minimizer for $\cc_d(V_n,B_n,L_n)$. For any subsequence 
	$\{X_{n_k}\}$ of $\{X_n\}$, if $\HM^d\mr (X_{n_k}\setminus B_{n_k})\wc \mu$,
	then for any $z\in H \setminus B$, $\Theta_{\mu}^d(z,\cdot)$ is
	non-decreasing and $\Theta_{\mu}^d(z)\geq 1$ for every $z\in \spt\mu
	\setminus B$.
\end{lemma}
\begin{proof}
	For any $x\in V_n\setminus B_n$ and $0<r<\dist(x,V_n\setminus B_n)$, we put
	$\Theta_n^d(x,r)=\Theta_{\mu_n}^d(x,r)$ and
	$\Theta_n^d(x)=\Theta_{\mu_n}^d(x)$, where $\mu_n=\HM^d\mr (X_n\setminus B_n)$.  
	Since $X_n$ is a minimizer for $\cc_d(V_n,B_n,L_n)$, we get that $X_n\setminus
	B_n$ is minimal in $V_n\setminus B_n$. By Proposition \ref{prop:de}, we get
	that the function $\Theta_n^d(x,\cdot)$ is non-decreasing for every $x\in
	V_n\setminus B_n$. Since $X_n\setminus B_n$ is 
	minimal, thus $X_n\setminus B_n$ is $d$-rectifiable, and for $\HM^d$-a.e. 
	$x\in X_n\setminus B_n$, $\Theta_n^d(x)= 1$.

	For any $z\in \spt\mu \setminus B$, we can find a sequence $\{z_{n_k}\}$ with
	$z_{n_k}\in X_{n_k} $ such that $\|z-z_{n_k}\|\to 0$ and $\Theta_n^d(z_{n_k})
	=1$. Thus, for any $0<r<r_z=\dist(z,H\setminus B)$, setting 
	\[
		r_k=\min\{r-
		\|z-z_{n_k}\|, \dist(z_{n_k},V_{n_k}\setminus B_{n_k})\},
	\]
	we have that 
	$r_k\to r$ and 
	\[
		\mu(\cball(z,r)) \geq \limsup_{k\to \infty}\HM^d(E_{n_k}\cap \cball(z,r))
		\geq \limsup_{k\to \infty}\HM^d(E_{n_k}\cap \cball(z,r_k))\geq
		\limsup_{k\to \infty}\omega_d r_k^d = \omega_d r^d.
	\]
	Hence $\Theta_{\mu}^d(z)\geq 1$ for every $z\in \spt\mu\setminus B$.

	For any $z\in H\setminus B$ and $0<s<t<\dist(z,H\setminus B)$, setting 
	$r_{k,z} = \dist(z,V_n\setminus B_n) $, then $r_{k,z}\to r_z$, thus
	$t<r_{k,z}$ for $k$ large enough.  If $\mu(\partial
	\cball(z,s))=\mu(\cball(z,t))=0$, then 
	\[
		\mu(\cball(z,t)) = \lim_{k\to \infty}\HM^d(E_{n_k}\cap \cball(z,t)) 
	\]
	and 
	\[
		\mu(\cball(z,s)) = \lim_{k\to \infty}\HM^d(E_{n_k}\cap \cball(z,s)). 
	\]
	Hence 
	\[
		\Theta_{\mu}^d(z,t)= \lim_{k\to \infty} \Theta_{n_k}^d(z,t)\geq \lim_{k\to
		\infty} \Theta_{n_k}^d(z,s) = \Theta_{\mu}^d(z,s).
	\]
	Since $\mu(\partial \cball(z,s))=0$ for $\HM^1$-a.e. $s\in (0,r_z)$, for any
	$0<r<R<\dist(z,H\setminus B)$, we can find $s$ and $t$ with $r<s<t<R$ and
	$\mu(\partial \cball(z,s))=\mu(\cball(z,t))=0$, thus
	\[
		\Theta_{\mu}^d(z,r)\leq \Theta_{\mu}^d(z,s)\leq \Theta_{\mu}^d(z,t)\leq
		\Theta_{\mu}^d(z,R).
	\]
	Hence $\Theta_{\mu}^d(z,\cdot)$ is non-decreasing.
\end{proof}

\begin{lemma}
	Let $X_n$ be a minimizer for $\cc_d(V_n,B_n,L_n)$ such that $\spt(\HM^d\mr
	(X_n\setminus B)) \supseteq X_n\setminus B_n$. For any subsequence 
	$\{X_{n_k}\}$ of $\{X_n\}$, if $\HM^d\mr (X_{n_k}\setminus B_n)\wc \mu$, 
	then $X_{n_k}\to B\cup \spt\mu$ in Hausdorff distance.
\end{lemma}
\begin{proof}
	Suppose that $\{X_{n_k}\}$ is a subsequence of $\{X_n\}$ such that $\HM^d\mr
	(X_{n_k}\setminus B_{n_k}) \wc \mu$. We put $X=B\cup \spt \mu$. Let
	$\{Y_m\}$ be a subsequence of $\{X_{n_k}\}$ such that $Y_m$ converges to a
	compact set $Y$ in Hausdorff  distance. Then $X\subseteq Y$. For any $y\in
	Y\setminus B$, and $0<r<r_y=\dist(y,H\setminus B)$, we take $y_m\in Y_m$
	such that $y_m\to y$, and put $r_m= r- \|y-y_m\|$, then 
	\[
		\mu(\cball(y,r))\geq \limsup_{m\to \infty}\HM^d(Y_m\cap \cball(y,r))
		\geq \limsup_{m\to \infty}\HM^d(Y_m\cap \cball(y_m,r_m))\geq \limsup_{m\to
		\infty}\omega_d r_m^{d} = \omega_d r^d.
	\]
	Thus $y\in \spt\mu$, and we get that $Y\subseteq X$. Hence $X_{n_k}\to X$ in
	Hausdorff distance.

\end{proof}
\begin{lemma}
	Let $\{E_n\}$ be a sequence of compact subsets in $H$ with $E_n\in \cc_d(V_n,
	B_n,L_n)$. If $E_n$ converges to $E$ in Hausdorff distance, then $E\in
	\cc_d(H,B,L)$. 
\end{lemma}
\begin{proof}
We claim that, for any compact set $F\subseteq H$ with $F\supseteq B$, if
$\pi_n(F)\in \cc_d(V_n,B_n,L_n)$, then $F\in \cc_d(H,B,L)$. Indeed, we put $F_n=
\pi_n(F)$, $f_{m,n}=\pi_n\vert_{F_m}$ and $g_{m,n}=\pi_n\vert_{G_m}$ for
$m\geq n$. Then $((F_n)_{n\geq 1},(f_{m,n})_{m\geq n})$ and $((B_n)_{n\geq 1},
(g_{m,n})_{m\geq n})$ are inverse systemes. And $B=\varprojlim B_n$ and
$F=\varprojlim F_n$. Since $i_{B_n,F_n}\circ g_{m,n} = f_{m,n}\circ i_{B_m,F_m}
$ for any $m\geq n$, we get that 
\[
	i_{B,F} = \varprojlim i_{B_n,F_n}
\]
and 
\[
	\cech_{d-1}(i_{B,F})= \varprojlim \cech_{d-1}(i_{B_n,F_n}).
\]
Since $F_n\in \cc_d(V_n,B_n,L_n)$, we have that $\cech_{d-1}(i_{B_n,F_n})(L_n)=0$,
thus  $\cech_{d-1}(i_{B,F})(L)= 0$, and $F\in \cc_d(H,B,L)$.

For any $k\geq 1$, we put $X_k= \overline{\cup_{n\geq k} E_n}$. For any $m\geq
k$, since $E_m\subseteq \pi_m(X_k)$, we see that $\pi_m(X_k)\in \cc_d(V_m,B_m,
L_m)$. For any $1\leq m\leq k$, $\pi_n = \pi_n \circ \pi_k$, thus  $\pi_m(X_k)
=\pi_m(\pi_k(X_k))$ and 
\[
	L_n = \cech_{d-1}(\pi_n)(L) = \cech_{d-1}(\pi_n)\circ \cech_{d-1}(\pi_k)(L)
	= \cech_{d-1}(\pi_n)(L_k).
\]
Since $\pi_k(X_k)\in \cc_d(V_k,B_k,L_k)$, we have that $\pi_m(X_k)\in \cc_d(V_m,
B_m,L_m)$. We get so that $\pi_m(X_k)\in \cc_d(V_m, B_m,L_m)$ for any $m\geq 1$.
Hence $X_k\in \cc_d(H,B,L)$. 

Since $B\subseteq X_{k+1}\subseteq X_k$ for any $k\geq 1$, and $\cap_{k\geq 1}
X_n = E$, we get that $((X_n)_{n\geq 1},(i_{X_m,X_n})_{m\geq n})$ is an inverse
system, and $E= \varprojlim X_n$. Since $i_{B,X_n} = i_{X_m,X_n}\circ i_{B,X_m}
$, we get that 
\[
	i_{B,E} = \varprojlim i_{B,X_n},
\]
and 
\[
	\cech_{d-1}(i_{B,E})= \varprojlim \cech_{d-1}(i_{B,X_n}).
\]
Since $X_n\in \cc_d(H,B,L)$, we get that $\cech_{d-1}(i_{B,X_n})(L)=0$. Hence
$\cech_{d-1}(i_{B,E})(L)=0$, and $E\in \cc_d(H,B,L)$.
\end{proof}
\begin{lemma}\label{le:li}
	Let $X_n$ be a minimizer for $\cc_d(V_n,B_n,L_n)$. Then 
	\[
		\limsup_{n\to \infty}\HM^d(X_n\setminus B_n)\leq \inf\{\HM^d(E\setminus B)
		:E\in \cc_d(H,B,L)\}.
	\]
\end{lemma}
\begin{proof}
	For any $\delta>0$, the function $\eta_{\delta}$ defined by 
	\[
		\eta_{\delta}(t)= \begin{cases}
			0, & t\leq 0,\\
			t/\delta, & 0<t<\delta,\\
			1,& t\geq \delta
		\end{cases}
	\]
	is a Lipschitz function. Let $\varphi_{n,\delta}:H\to H$ be the mapping
	defined by 
	\begin{equation}\label{eq:df}
		\varphi_{n,d}(x)=\eta_{\delta}(\dist(x,B))\pi_n(x)
		+(1-\eta_{\delta}(\dist(x,B)))x.
	\end{equation}
	We put $B^{\delta}= B+\oball(0,\delta)$ and $\pi_n^{\perp}= \id -\pi_n$.
	Then $\|\varphi_{n,\delta} -\id\|_{\infty}\leq \delta_n$, 
	$\varphi_{n,\delta}\vert_B =\id_B$ and $\varphi_{n,\delta}\vert_{H\setminus
	B^{\delta}}= \id_{H\setminus B^{\delta}}$, and any $x,y\in H$, 
	\[
		\begin{aligned}
			\varphi_{n,\delta}(x)-\varphi_{n,\delta}(y)&= (x-y)-\left(\eta_{\delta}
			(\dist(x,B))\pi_n^{\perp}(x)-\eta_{\delta}(\dist(y,B))\pi_n^{\perp}(y)\right)\\
			&= (x-y)- \eta_{\delta}(\dist(x,B))\left(\pi_n^{\perp}(x)-
			\pi_n^{\perp}(x)\right)\\
			&\quad -\left(\eta_{\delta}(\dist(x,B))-\eta_{\delta}
			(\dist(y,B))\right)\pi_n^{\perp}(y),
		\end{aligned}
	\]
	thus
	\[
		\|\varphi_{n,\delta}(x)-\varphi_{n,\delta}(y)\|\leq \|x-y\|+ \delta_n
		\Lip(\eta_{\delta})|\dist(x,B)-\dist(y,B)|\leq (1+\delta_n/\delta)\|x-y\|.
	\]
	Hence $\varphi_{n,\delta}$ is Lipschitz and $\Lip(\varphi_{n,\delta})\leq
	1+\delta_n/\delta$. 

	For any $E\in \cc_d(H,B,L)$, setting $E_{n,\delta}=\varphi_{n,\delta}(E)$, we
	have that $E_{n,\delta} \in \cc_d(H,B,L) $ and $\pi_n(E_{n,\delta})= \pi_n(E)
	\in \cc_d(V_n,B_n,L_n)$. Hence 
	\[
		\HM^d(E_{n,\delta}\setminus B) \geq \HM^d(\pi_n(E_{n,\delta}\setminus B))
		\geq \HM^d(\pi_n(E_{n,\delta})\setminus B_n)
		\geq \HM^d(\pi_n(E)\setminus B_n)\geq \HM^d(X_n\setminus B_n).
	\]
	But 
	\[
		\begin{aligned}
			\HM^d(E_{n,\delta}\setminus B)&= \HM^d(\varphi_{n,\delta}(E\setminus B))
			= \HM^d(\varphi_{n,\delta}
			(E\cap B^{\delta}\setminus B))+ \HM^d(\pi_n(E\setminus B^{\delta}))\\
			&\leq (1+\delta_n/\delta)^d\HM^d(E\cap B^{\delta}\setminus B)
			+\HM^d(E\setminus B^{\delta}),
		\end{aligned}
	\]
	we get that 
	\[
		\limsup_{n\to \infty}(\HM^d(X_n\setminus B_n))\leq \limsup_{n\to \infty}
		(1+\delta_n/\delta)^d\HM^d(E\cap B^{\delta}\setminus B)
		+\HM^d(E\setminus B^{\delta}) = \HM^d(E\setminus B).
	\]
	Hence 
	\[
		\limsup_{n\to \infty}\HM^d(X_n\setminus B_n)\leq \inf\{\HM^d(E\setminus B)
		:E\in \cc_d(H,B,L)\}.
	\]
\end{proof}
\begin{proof}[Proof of Theorem \ref{thm:PP}]
	Let $X_n$ be a minimizer for $\cc_d(V_n,B_n,L_n)$. We put $\mu_n=\HM^d\mr (X_n
	\setminus B_n)$, assume $\mu_{n_k}\wc \mu$, and put $X=B\cup \spt\mu$. Then
	$X_{n_k}$ converges to $X$ in Hausdorff distance, thus $X\in \cc_d(H,B,L)$. 
	By the lemma above, we have that 
	\[
		\begin{aligned}
			\mu(X\setminus B)&=\mu(H\setminus B)\leq \liminf_{k\to \infty}\mu_{n_k}
			(H\setminus B)\leq \liminf_{k\to \infty}\HM^d(X_n\setminus B_n)\\
			&\leq \inf\{\HM^d(E\setminus B): E\in \cc_d(H,B,L)\}\leq \HM^d(X\setminus B).
		\end{aligned}
	\]
	By Lemma \ref{le:md}, we have $\Theta_{\mu}^d(x)\geq 1$ for $\HM^d$-a.e.
	$x\in X$. Thus $\HM^d(X\setminus B)\leq \mu(X\setminus B)$.
	Hence $X$ is a minimizer for $\cc_d(H,B,L)$.
\end{proof}
\bibliographystyle{plain}
\bibliography{gmt}

\end{document}